\documentclass[12pt]{amsart}
\usepackage{amssymb,amsmath,verbatim,enumerate,ifthen}
\usepackage{cite}
\usepackage[mathscr]{eucal}
\usepackage[utf8]{inputenc}
\usepackage[T1]{fontenc}
\usepackage{esint}
\usepackage[marginparwidth=2.4cm]{geometry}

\DeclareUnicodeCharacter{1EF3}{\`{y}}

\makeatletter
\@namedef{subjclassname@2010}{%
  \textup{2010} Mathematics Subject Classification}
\makeatother

\newtheorem{thm}{Theorem}
\newtheorem{cor}{Corollary}
\newtheorem{prop}{Proposition}

\theoremstyle{remark}
\newtheorem{rem}{Remark}[section]

\newtheorem*{exa}{Example}

\numberwithin{equation}{section}

\newcommand\nolabel[1]{\nonumber}

\def\G{\mathscr{G}}
\def\Hc{\mathscr{H}}
\def\P{\mathscr{P}}
\newcommand\R{\mathbb{R}}
\newcommand\N{\mathbb{N}}
\newcommand\Q{\mathbb{Q}}
\newcommand\Z{\mathbb{Z}}

\newcommand{\QA}[2][SKIPPED]{
    \ifthenelse{\equal{#1}{SKIPPED}}
        {\mathscr{A}^{[#2]}}
        {\mathscr{A}^{[#2]}|_{#1}}
}

\newcommand{\norm}[1]{\left\| #1 \right\| }

\newcommand{\ceil}[1]{\left\lceil #1 \right\rceil}

\DeclareMathOperator{\CM}{\mathcal{CM}}
\DeclareMathOperator{\M}{\mathscr{M}}
\numberwithin{equation}{section}
\def\eq#1{{\rm(\ref{#1})}}
\def\Eq#1#2{\ifthenelse{\equal{#1}{*}}
  {\begin{equation*}\begin{aligned}[]#2\end{aligned}\end{equation*}}
  {\begin{equation}\begin{aligned}[]\label{#1}#2\end{aligned}\end{equation}}}

\newcommand{\operator}[1]{\mathop{\vphantom{\sum}\mathchoice
{\vcenter{\hbox{\LARGE $#1$}}}
{\vcenter{\hbox{\Large $#1$}}}{#1}{#1}}\displaylimits}
\def\AM#1{\operator{\mathscr{A}^{\!\mbox{\footnotesize $[#1]$}}}}       
\def\PM#1{\operator{\mathscr{P}_{\!\mbox{\footnotesize $#1$}}}}

\def\Mm{\operator{\mathscr{M}}}

\title{Weakening of Hardy property for means}
    \subjclass[2010]{26E60, 26D15}
\keywords{Hardy means, Power means, Gini means, quasi-arithmetic means, Gaussian product, invariant means, bounded operators, selfmapping of $\ell_1$}

\author[P. Pasteczka]{Pawe\l{} Pasteczka}
\address{Institute of Mathematics \\ Pedagogical University of Cracow \\ Podchor\k{a}\.zych str.~2, 30-084 Krak\'ow, Poland}
\email{pawel.pasteczka@up.krakow.pl}

\begin{document}
\maketitle
\begin{abstract}
 The aim of this paper is to find a broad family of means defined on a subinterval of $I \subset [0,+\infty)$ such that
 $$
 \sum_{n=1}^\infty \mathscr{M}(a_1,\dots,a_n) <+\infty \quad\text{ for all }\quad a \in \ell_1(I).$$
 Equivalently, the averaging operator 
$$
(a_1,\,a_2,a_3\,,\dots) \mapsto \big( a_1,\,\mathscr{M}(a_1,a_2),\,\mathscr{M}(a_1,a_2,a_3), \dots\big)
$$
is a selfmapping of $\ell_1(I)$.

 This property is closely related to so-called Hardy inequality for means (which additionally requires boundedness of this operator). In fact we prove that these two properties are equivalent in a family of Gini means and Gaussian product of Power means. Moreover it is shown that this is not the case for quasi-arithmetic means.
\end{abstract}

\section{Introduction}
A mean $\M$ on an interval $I \subset [0,+\infty)$ (that is a function $\M \colon \bigcup_{n=1}^\infty I^n \to I$ satisfying $\min(a) \le \M(a) \le \max(a)$ for every admissible vector $a$) is said to be a \emph{Hardy mean} if there exists a finite constant $C$ such that
\Eq{*}{
\sum_{n=1}^\infty \Mm_{k=1}^n (a_k) \le C \sum_{n=1}^\infty a_n \quad \text{ for all }\quad a \in \ell_1(I),
}
where $\Mm_{k=1}^n(a_k)$ stands for $\M(a_1,\dots,a_n)$. The smallest extended real number $C$ satisfying the inequality above is called a \emph{Hardy constant} of $\M$ and denoted by $\Hc(\M)$.

These definitions were introduced recently by P\'ales-Persson \cite{PalPer04} and P\'ales-Pasteczka \cite{PalPas16}, respectively. In fact they are closely related as a mean is Hardy if and only if its Hardy constant is finite.

On the other hand there are a number of earlier result which can be expressed in terms of Hardy mean and Hardy constant. 
These properties were studied for power means $\P_\alpha$ in a series of papers \cite{Har20,Lan21,Car32,Kno28}. Their result (in a unified form) can be expressed as 
\Eq{*}{
\Hc(\P_\alpha)=
\begin{cases} 
(1-\alpha)^{-1/\alpha}&\alpha \in (-\infty,0) \cup (0,1), \\ 
e & \alpha=0, \\
+\infty & \alpha\in[1,\infty).
\end{cases} 
}

More about the history of the developments related to Hardy-type inequalities is sketched in  surveys by Pe\v{c}ari\'c--Stolarsky 
\cite{PecSto01}, Duncan--McGregor \cite{DucMcG03}, and in a book of 
Kufner--Maligranda--Persson \cite{KufMalPer07}.
Further examples of Hardy means (with known Hardy constant) were given recently by Pasteczka \cite{Pas15b} and P\'ales-Pasteczka \cite{PalPas16,PalPas18a}. Some negative results were obtained in \cite{Pas15c} (see Proposition~\ref{prop:Pas15c} below).

Let us emphasize that Hardy property of a mean $\M$ on $I$ can be expressed in terms of \emph{$\M$-averaging operator} defined by
\Eq{*}{
I^\N \ni (a_1,a_2,\dots) \mapsto\big( a_1,\,\M(a_1,a_2),\, \M(a_1,a_2,a_3),\,\dots\big) \in I^\N\:.
}
Indeed, a mean $\M$ is a Hardy mean if and only if $\M$-averaging operator is a bounded operator from $\ell_1(I)$ to itself. In fact its norm equals $\Hc(\M)$.
Motivated by these preliminaries we will be dealing with a more general definition. Namely, we call a mean $\M$ on $I$ to be a \emph{weak-Hardy mean} if
\Eq{def:wH}{
\sum_{n=1}^\infty \Mm_{k=1}^n (a_k) < +\infty \quad \text{ for all }\quad a \in \ell_1(I).
}
Equivalently, the $\M$-averaging operator is a selfmap of $\ell_1(I)$ (with no boundedness assumption).

\begin{rem}
 Let us observe that $\M$-averaging operator is a selfmapping of $\ell_1(I)$ if and only if the conjugated operator
 \Eq{*}{
 (I^{1/p})^\N \ni (a_1,a_2,a_3,\dots) \mapsto \big( a_1, \M(a_1^p,a_2^p)^{1/p},\M(a_1^p,a_2^p,a_3^p)^{1/p},\dots \big)
 }
 is a selfmapping of $\ell_p(I^{1/p})$; $p \in(1,\,+\infty)$.
 
 In this way the consideration in the present paper can be easily generalized to $\ell_p$ spaces.
 \end{rem}

Obviously, each Hardy mean is a weak-Hardy mean, but in general the converse implication is not valid. In this manner we are interested in families of means where all weak-Hardy means are Hardy one. For example it is easy to verify that it is the case for power means. We prove this property for Gini means (section~\ref{sec:Gini}). On the other hand we show that it is not the case for quasi-arithmetic means (section~\ref{sec:QA}). In general this problem remains open (compare with Remark~\ref{rem:equiv}).

We conclude this paper with some results in a family of quasi-arithmetic means. In particular we prove that in this case weak-Hardy property is determined by values of mean in a neighbourhood of zero.

\subsection{Basic properties of means}
Based on \cite{PalPas16}, let us recall some notions. We say that $\M$ is \emph{symmetric} and \emph{(strictly) increasing} if for all 
$n\in\N$ the $n$-variable restriction $\M|_{I^n}$ is symmetric and (strictly) increasing in each of its variables, respectively. If $I=\R_+$, we can analogously define the notion of homogeneity of $\M$. \emph{Monotonicity} of mean is associated with its increasingness. Finally, the mean $\M$ is called \emph{repetition invariant} if for all $n,m\in\N$
and $(a_1,\dots,a_n)\in I^n$ the following identity is satisfied
\Eq{*}{
  \M(\underbrace{a_1,\dots,a_1}_{m\text{-times}},\dots,\underbrace{a_n,\dots,a_n}_{m\text{-times}})
   =\M(a_1,\dots,a_n).
}


\section{Necessary conditions for weak-Hardy property}

In this section we deliver some necessary condition for the mean to be weak-Hardy. First, we go back to the paper \cite{Pas15c}, where such a result was obtained for Hardy property. 
\begin{prop}[\!\cite{Pas15c}, Theorem 1.1]
\label{prop:Pas15c}
Let $I \subset \R_{+}$ be an interval, $\inf I=0$. Let $\M$ be a mean defined on $I$ and $(a_n)_{n=1}^{\infty}$ be a sequence of numbers in $I$ satisfying $\sum_{n=1}^\infty a_n = +\infty$. If\, $\lim\limits_{n \to \infty} a_n^{-1} \Mm_{k=1}^n(a_k) = +\infty$ then $\M$ is not a Hardy mean.
\end{prop}

Our aim is to establish an analogue of this result for weak-Hardy property. Nevertheless, we need to introduce some technical notation first.

We say that a sequence $(a_n)$ of positive numbers is \emph{nearly increasing} if there exists $\varepsilon>0$ such that for every $m,\,n \in\N$ with $m \le n$ we have $\varepsilon a_m \le a_n$. Notice that nearly increasing sequences inherit some properties which are characteristic for monotone sequences. For example is easy to verify that every such sequence is either divergent or bounded (in fact $\liminf a_n \ge \varepsilon\sup a_n$). On the other hand, bounded sequence is nearly increasing if and only if it is separated from zero. Therefore, this definition is meaningful mostly for divergent sequences.

Having this already introduced our main result reads as follows

\begin{thm}\label{thm:main}
Let $\M$ be a homogeneous and monotone mean defined on $\R_+$. If there exists a sequence $(a_n)$ of positive numbers such that
 \begin{enumerate}
  \item $\sum_{n=1}^\infty a_n=+\infty$, 
  \item a sequence $(a_n^{-1}\Mm_{k=1}^n(a_k))_{n=1}^\infty$ is nearly increasing and divergent,
  \item $\sum_{n=1}^\infty a_n^{1+s} \big(\Mm_{k=1}^n(a_k)\big)^{-s}$ is finite for some $s \in \R_+$,
 \end{enumerate}
then $\M$ is not a weak-Hardy mean.
\end{thm}
\begin{proof}
 Let $b_n:=a_n^{-1} \Mm_{k=1}^n(a_k)$ and $\varepsilon>0$ be the parameter which appears in the definition of nearly increasingness in the second assumption. We can rewrite (3) in a compact form
\Eq{E:Conv-s}{
\sum_{n=1}^{\infty} a_n b_n^{-s} <+\infty \quad \text{ for some } s>0.
}

Then, as $\M$ is homogeneous, monotone and $(b_n)$ is nearly increasing, we have 
\Eq{E:IAS}{
\sum_{n=1}^{\infty} \Mm_{k=1}^n (a_kb_k^{-s}) 
&\ge \sum_{n=1}^{\infty} \Mm_{k=1}^n (a_k\varepsilon^sb_n^{-s})
= \sum_{n=1}^{\infty} \varepsilon^sb_n^{-s} \Mm_{k=1}^n a_k 
= \sum_{n=1}^{\infty} \varepsilon^sa_n b_n^{1-s}.
}
We prove this theorem by induction with respect to $s$ (or, more precisely, with respect to $\ceil s$ ).

For $s \in (0,1]$, by $\sum_{n=1}^{\infty}a_n=+\infty$; $\lim_{n \to \infty} b_n=+\infty$, and \eq{E:IAS} we get
\Eq{*}{
\sum_{n=1}^{\infty} \Mm_{k=1}^n (a_kb_k^{-s})=+\infty.
}
By property \eq{E:Conv-s} we obtain that $\M$ is not a weak-Hardy mean.

For $s>1$ we obtain that either the sum on the most right hand side of \eq{E:IAS} is infinite and, consequently, $\M$ does  not admit weak-Hardy property or
\Eq{*}{
\sum_{n=1}^\infty a_n b_n^{1-s}<+\infty,
}
which is exactly the third condition with $s$ replaced by $s-1$. By inductive assumption $\M$ is not a weak-Hardy mean in this case too.
\end{proof}

In the special case $a_n=\tfrac1n$ and arbitrary positive $D$ ($D=2/s$), Theorem~\ref{thm:main} implies

\begin{cor}
\label{cor:1}
Let $\M$ be a homogeneous and monotone mean defined of $\R_+$. If $\big(\Mm_{k=1}^n \big(\tfrac nk\big)\big)_{n=1}^\infty$ is nearly increasing, and there exist $C,\,D\in \R_+$ and $n_0 \in \N$ such that
\Eq{E:Le}{
\Mm_{k=1}^n\big(\tfrac1k\big) \ge \frac{C (\ln n)^D}n \qquad \text{for all }n \ge n_0,
}
then $\M$ is not a weak-Hardy mean.
\end{cor}

At this place nearly increasingness of the sequence $(\Mm_{k=1}^n \big(\tfrac nk\big))_{n=1}^\infty$ seamed to be the most restrictive condition. Luckily we have the following 
\begin{prop}
Let $\M$ be a homogeneous, monotone, and repetition invariant mean defined of $\R_+$. Then the sequence $(\Mm_{k=1}^n \big(\tfrac nk\big))_{n=1}^\infty$ is nearly increasing (with $\varepsilon=\tfrac12$).
\end{prop}
\begin{proof}
Let $d_n:=\Mm_{k=1}^n \big(\tfrac nk\big)$.  
 We prove that $d_m \le 2 d_n$ for all $m \le n$. The proof is divided into two parts
 \begin{enumerate}
  \item[(i)] $d_p \le 2 d_q$ for $p \in \N$ and $q \in \{p,\dots,2p-1\}$,
  \item[(ii)] $d_p \le d_{2p}$ for $p \in \N$.
 \end{enumerate}
Then we can use simple induction to obtain the final assertion.

As the first inequality for $p=q$ is trivial, fix $p \in \N$ and $q \in \{p+1,\dots,2p-1\}$. Consider two sequences of length $pq$:
\Eq{*}{
a=\Big(\frac q{\ceil{k/p}} \Big)_{k=1}^{pq}\qquad\text{ and }\qquad b=\Big(\frac p{\ceil{k/q}} \Big)_{k=1}^{pq}.
}
For $k \le p$ we get $a_k/b_k=\tfrac qp\ge \tfrac12$. Similarly for $k>p$, by $\ceil{k/p} \le 2k/p$, we get
\Eq{*}{
\frac{a_k}{b_k}=\frac{q\ceil{k/q}}{p\ceil{k/p}} \ge \frac{q \cdot k/q}{p \cdot 2k/p}=\frac12.
}
Consequently $b_k \le 2 a_k$ for all $k \in \{1,\dots,pq\}$. Thus, by monotonicty, homogeneity, and repetition invariance of $\M$, we get
\Eq{*}{
d_p=\Mm_{k=1}^p \big(\tfrac pk\big)=\Mm_{k=1}^{pq} b_k \le 
2\Mm_{k=1}^{pq} a_k =
2\Mm_{k=1}^{q} \big(\tfrac qk\big)=2d_q\:,
}
which is (i). 

The second inequality is significantly simpler. Indeed, for every $p \in \N$ we simply obtain
\Eq{*}{
d_p=\Mm_{k=1}^p \big(\tfrac pk\big)=\Mm_{k=1}^{2p} \Big(p \ceil{\tfrac k2}^{-1}\Big)
\le \Mm_{k=1}^{2p} \Big(p \cdot \big(\tfrac k2\big)^{-1}\Big)
= \Mm_{k=1}^{2p} \big(\tfrac {2p}k\big)=d_{2p}\:.
}
At the moment define $s \in \N \cup \{0\}$ and $\theta \in [1,2)$ such that $n= 2^s\theta m$. Then, applying (ii) iteratively and then (i), we obtain
\Eq{*}{
d_m \le d_{2m} \le \dots \le d_{2^sm} \le 2d_{2^s\theta m}=2d_n,
}
what was to be proved.
\end{proof}

Binding this result with Corollary~\ref{cor:1}, we obtain 
\begin{cor} \label{cor:2}
Let $\M$ be a homogeneous, monotone, and repetition invariant mean. If there exist $C,\,D\in \R_+$ and $n_0 \in \N$ such that condition \eq{E:Le} is valid, then $\M$ is not a weak-Hardy mean.
\end{cor}

\section{Applications}
In the subsequent sections we discuss a weak-Hardy property among several families of means.

\subsection{\label{sec:QA}Quasi-arithmetic means}
Quasi-arithmetic means were introduced in series of several simultaneous papers \cite{Kno28,Def31,Kol30,Nag30} in 1920-s/30-s as a generalization of already mentioned family of power means. 
For a continuous and strictly monotone function $f \colon I \to \R$ (hereafter $I$ is an interval and $\CM(I)$ stands for a family of all continuous and monotone functions of $I$) and a vector $a=(a_1,a_2,\dots,a_n) \in I^n$, $n \in \N$ we define 
\Eq{*}{
\QA{f}(a):=f^{-1}\left( \frac{f(a_1)+f(a_2)+\cdots+f(a_n)}{n} \right).
}

For a subinterval $J \subset I$ we denote by $\QA[J]{f}$ the restriction of quasi-arithmetic mean to an interval $J$, i.e. $\QA[J]{f}:= \QA{f}|_{\bigcup_{n=1}^\infty J^n}$.
It is easy to verify that for
$I=\R_+$ and $f=\pi_p$, where $\pi_p(x):=x^p$ if $p\ne 0$ and $\pi_0(x):=\ln x$, the mean $\QA{f}$ coincides with the $p$-th power mean. 

Hardy property for this family was characterized by Mulholland \cite{Mul32} shortly after its formal definition. He proved that $\QA{f}$ is a Hardy mean if and only if there exist $\alpha<1$, and $C>0$ such that $\QA{f}(a) \le C \cdot \P_\alpha(a)$ for every $a \in \bigcup_{n=1}^{\infty} I^n$. Now we turn into weak-Hardy property. First, let us present a result which provides localizability of weak-Hardy property for quasi-arithmetic means.
\begin{thm} \label{thm:wHwH}
Let $I$ be an interval with $\inf I=0$, and $f \in \CM(I)$. If there exist $\varepsilon \in I$, such that $\QA[(0,\varepsilon)]{f}$ is a weak-Hardy mean, then $\QA{f}$ is a weak-Hardy mean.
\end{thm}

Second, let us establish much stronger result under a bit different assumptions.
\begin{thm} \label{thm:whQA}
Let $I$ be an interval with $\inf I=0$, and $f \in \CM(I)$.
If there exist $\varepsilon \in I$, 
such that $\QA[(0,\varepsilon)]{f}$ is a Hardy mean, then 
there exists a function $c_f \colon I \to \R_+$ such that
\Eq{*}{
\sum_{n=1}^\infty \: \AM{f}_{k=1}^n (a_k) \le c_f\big(\norm{a}_\infty\big) \cdot \norm{a}_1\quad \text{ for all } \quad a \in \ell_1(I).
}
\end{thm}

Proofs of these theorems are postponed until section~\ref{sec:wHwH} and \ref{sec:whQA}, respectively. 

In fact our conjecture is that weak-Hardy property of quasi-arithmetic mean is equivalent to the fact that its restriction to some interval $(0,\varepsilon)$ (for $\varepsilon \in I$\:) is a Hardy mean. It is worth mentioning that this property does not depend on a choice of $\varepsilon$. More precisely we have the following result.
\begin{cor} \label{cor:whQA}
Let $I$ be an interval with $\inf I=0$, and $f \in \CM(I)$. If there exist $\varepsilon \in I$ such that $\QA[(0,\varepsilon)]{f}$ is a Hardy mean
then $\QA[(0,s)]{f}$ is a Hardy mean for all $s \in I$. 
\end{cor}

\begin{proof}
If $s\le \varepsilon$ the statement is trivial. From now on assume that $s>\varepsilon$.
By Theorem~\ref{thm:whQA} we know that there exists a constant $C:=c_f(s)$ such that
 \Eq{*}{
\sum_{n=1}^\infty \: \AM{f}_{k=1}^n (a_k) \le C \cdot \norm{a}_1\quad \text{ for all } \quad a \in \ell_1(I)\text{ with } \norm{a}_\infty=s.
 }

Now take $v \in \ell_1(0,s)$. If $\norm{v}_\infty \le \varepsilon$, then we have
\Eq{*}{
\sum_{n=1}^\infty \AM{f}_{k=1}^n(v_k) \le \Hc(\QA[(0,\varepsilon)]{f}) \sum_{n=1}^\infty v_n.
}

For $\norm{v}_\infty \in (\varepsilon,s]$, let us add the artificial element $v_0=s$. Then, as $v_0 \ge v_i$ for all $i \in \N$ we get
\Eq{*}{
\sum_{n=1}^\infty \AM{f}_{k=1}^n(v_k) &\le 
\sum_{n=1}^\infty \AM{f}_{k=0}^n(v_k) \le
\sum_{n=0}^\infty \AM{f}_{k=0}^n(v_k) \le c_f(s) \sum_{n=0}^\infty v_n=sc_f(s)+c_f(s) \sum_{n=1}^\infty v_n\\
&\le c_f(s) \cdot \frac{s}\varepsilon \cdot \sup_{n \in \{1,2,\dots\}} v_n + c_f(s) \sum_{n=1}^\infty v_n \le \big(1+\frac s\varepsilon\big) c_f(s) \sum_{n=1}^\infty v_n.
}
This yields that $\QA{f}$ restricted to $(0,s)$ is a Hardy mean with a Hardy constant majorized by
\Eq{*}{
\max \Big(\Hc(\QA[(0,\varepsilon)]{f}),\, \big(1+\frac s\varepsilon\big) c_f(s) \Big).
}
Thus the proof is ended.
\end{proof}

Let us conclude this section with a simple example that in a family of quasi-arithmetic means not every weak-Hardy mean is a Hardy mean.
\begin{exa}
 Let $f \colon (0,+\infty) \to \R$ be given by
 \Eq{*}{
 f(x):= \begin{cases} \ln x & \text{ if } x \in (0,1], \\ x-1 & \text{ if } x \in (1,\infty).
        \end{cases}
 }
 Obviously, as $\QA{f}$ restricted to $(0,1]$ is a geometric mean ($\P_0$), we get, by Theorem~\ref{thm:whQA}, that $\QA{f}$ is a weak-Hardy mean. We prove that it is not a Hardy mean. 
 
 Indeed, fix $N\in \N$ arbitrarily and define $a_n:=N^2/n^2$. Then we have 
 \Eq{E1:1}{
 \sum_{n=1}^{\infty} \AM{f}_{k=1}^n(a_k) \le 
 \Hc(\QA{f}) \cdot \sum_{n=1}^\infty a_n=
 N^2 \cdot \Hc(\QA{f}) \cdot \frac{\pi^2}{6}.
 }
On the other hand for all $n \le N$ we have $a_n \ge 1$ and,
as $\QA{f}$ restricted to $[1,\infty)$ coincide with arithmetic mean, we obtain
\Eq{*}{
\AM{f}_{k=1}^n(a_k)=\frac{1}{n}\sum_{k=1}^n a_k \ge \frac{a_1}{n}=\frac{N^2}{n} \qquad (n \le N).
}
Thus, using the well known estimation of harmonic sequence we get
\Eq{E1:2}{
 \sum_{n=1}^{\infty} \AM{f}_{k=1}^n(a_k) 
 \ge  \sum_{n=1}^{N} \AM{f}_{k=1}^n(a_k)  
 \ge  \sum_{n=1}^{N} \frac{N^2}{n}  
 \ge N^2 \ln N.
}
If we now combine \eq{E1:1} and \eq{E1:2} we obtain
$
N^2 \ln N \le N^2 \cdot \Hc(\QA{f}) \cdot \tfrac{\pi^2}{6}$,
which simplifies to 
$\Hc(\QA{f}) \ge \tfrac{6}{\pi^2} \ln N$.
Letting $N \to \infty$ we obtain $\Hc(\QA{f})=+\infty$, which proves that $\QA{f}$ is not a Hardy mean.
\end{exa}

\subsection{\label{sec:Gini}Gini means}
Another generalization of Power Means was proposed in 1938 by Gini (cf. \cite{Gin38}). {\it Gini means} is a two-parameters family defined of $\R_{+}$ by the equality ($p,\,q\in \R$)
$$\G_{p,q}(a_1,\ldots,a_n):=
\begin{cases}
\left(\frac{\sum_{i=1}^n a_i^p}{\sum_{i=1}^n a_i^q}\right)^{1/(p-q)} & \textrm{if\ }p \ne q\,,\\
\exp \left(\frac{\sum_{i=1}^{n} a_i^p \ln a_i}{\sum_{i=1}^{n}a_i^p}\right)& \textrm{if\ }p = q\,.
\end{cases}$$
For $q = 0$ one easily identifies here the $p$-th power mean. It is known that $\G_{p,q}=\G_{q,p}$ and Gini means are nondecreasing with respect to $p$ and $q$ (cf. e.g. \cite[p. 249]{Bul03}). Furthermore it was proved \cite{PalPer04,Pas15c} that
\Eq{*}{
\G_{p,q} \text{ is a Hardy mean } \iff \min(p,q) \le 0 \:\text{ and }\: \max(p,q)<1.
}
We prove that weak-Hardy and Hardy property coincide for Gini means that is
\begin{prop}
Gini mean is a weak-Hardy mean if and only if it is a Hardy mean. \end{prop}

\begin{proof}
 First, it was proved \cite{Pas15c} that for all $q<0$ a mean $\G_{1,q}$ satisfies inequality \eq{E:Le}. Furthermore, by the results of Losonczi \cite{Los71a,Los71c}, $\G_{p,q}$ is monotone if and only if $pq \le 0$. As both homogeneity and repetition invariance are easy to check therefore, by Corollary~\ref{cor:2}, we obtain that $\G_{1,q}$ is a weak-Hardy mean for no $q<0$.

Consequently, as for every $p \le p'$ and $q\le q'$ we have $\G_{p,q} \le \G_{p',q'}$ we have that $\G_{p,q}$ is not a weak-Hardy mean whenever $\max(p,q) \ge 1$. In other words,
\Eq{GME1}{
\G_{p,q} \text{ is a weak-Hardy mean } \Longrightarrow \max(p,q)<1.
}

At the moment suppose that $p,\, q \in (0,1)$, $p \ne q$. For $a_n:=2^{1-n}$ we have
\Eq{*}{
\G_{p,q}(a_1,\dots,a_n)=\left( \frac{1+2^{-p}+\dots+2^{(1-n)p}}{1+2^{-q}+\dots+2^{(1-n)q}}\right)^{1/(p-q)}=\left( \frac{1-2^{-np}}{1-2^{-nq}} \cdot \frac{1-2^{-q}}{1-2^{-p}}\right)^{1/(p-q)}
}
Thus 
\Eq{*}{
\lim_{n \to \infty} \G_{p,q}(a_1,\dots,a_n) = \left(\frac{1-2^{-q}}{1-2^{-p}}\right)^{1/(p-q)}>0.
}
This shows that $\G_{p,q}$ is not a weak-Hardy mean for $(p,q) \in (0,1)^2$, $p \ne q$. Moreover, for $p \in (0,1)$, easy-to-check inequality $\G_{p,p} \ge \G_{p,p/2}$ implies that $\G_{p,p}$ in not a weak-Hardy mean too. 
These facts jointly with \eq{GME1} yield
\Eq{*}{
\G_{p,q} \text{ is weak-Hardy} 
\Longrightarrow \big( \min(p,q) \le 0 \text{ and } \max(p,q)<1 \big) \Longrightarrow \G_{p,q}\text{ is Hardy}.
}
As the converse implication is trivial, the proof is complete.
\end{proof}

\begin{rem}
We can use the same argumentation to prove that the Gaussian product of Power means has the same property (cf. \cite[section~3.1]{Pas15c} and Corollary~\ref{cor:2} above)
\end{rem}

\begin{rem}\label{rem:equiv}
 It remains an open question how to verify equivalence of Hardy and weak-Hardy property without verifying these properties separately.
\end{rem}

\section{Proofs of Theorem~\ref{thm:wHwH} and Theorem~\ref{thm:whQA}}

At the very beginning of this section let us underline that Theorem~\ref{thm:wHwH} and Theorem~\ref{thm:whQA} are closely related, however none of them is a consequence of the second one.
In fact Theorem~\ref{thm:whQA} provides stronger statement under a more restrictive assumptions. This motivates us to bind two proofs together in a rather unconventional way. 

In the first subsection we provide the proof of Theorem~\ref{thm:wHwH}. Later, we will prove Theorem~\ref{thm:whQA}. However, as  this theorem has a stronger assumptions, all intermediate steps and notations which will be made in section~\ref{sec:wHwH} remain valid in section~\ref{sec:whQA}. Consequently, we may refer to them as it would be an immanent part of the proof.
\subsection{\label{sec:wHwH}Proof of Theorem~\ref{thm:wHwH}}
Take $a \in \ell_1(I)$ arbitrarily.
 If we define $f(0):=\lim_{x \to 0^+} f(x) \in [-\infty,+\infty]$, then, applying Cesaro limit principle, and continuouity of $f^{-1}$ on $f(I \cup\{0\})$, we obtain, as $a_n \to 0$,
 \Eq{*}{
0&=f^{-1}(f(0))=f^{-1}\big(\lim_{n \to \infty} f(a_n)\big) =f^{-1}\Big(\lim_{n \to \infty} \frac{f(a_1)+\dots+f(a_n)}n\Big) \\
&=\lim_{n \to \infty} f^{-1}\Big(\frac{f(a_1)+\dots+f(a_n)}n\Big)
=\lim_{n \to \infty} \QA{f}(a_1,\dots,a_n).
}

Therefore let $n_0 \in \N$ be the smallest natural number such that 
\Eq{*}{
\QA{f}(a_1,\dots,a_n) \le \varepsilon\quad \text{ and } \quad a_n\le \varepsilon\qquad\text{ for all }n\ge n_0.
}

Define the sequence $(b_n)_{n=1}^\infty$ by 
\Eq{E:p0}{
b_n:=\begin{cases}
      \varepsilon & \text{ for } n \le n_0, \\
      a_n & \text{ for }n > n_0.
     \end{cases}
}
Then 
\Eq{E:15}{
\norm{b}_1 \le \varepsilon n_0+\norm{a}_1.
}

Moreover, as $\QA{f}$ is associative and monotone we obtain
\Eq{E:p1}{
\sum_{n=1}^\infty \: \AM{f}_{k=1}^n (a_k) 
&= \sum_{n=1}^{n_0} \: \AM{f}_{k=1}^n (a_k)+\sum_{n=n_0+1}^\infty \: \AM{f}_{k=1}^n (a_k)\\
&\le \sum_{n=1}^{n_0} \: \AM{f}_{k=1}^n (a_k)+\sum_{n=n_0+1}^\infty \: \AM{f}_{k=1}^n (b_k).
}
But $\norm{b_n}_\infty \le \varepsilon$ thus
\Eq{E:p2}{
\sum_{n=1}^\infty \: \AM{f}_{k=1}^n (a_k) \le n_0 \norm{a}_\infty +\sum_{n=1}^\infty \: \AM{f}_{k=1}^n (b_k)<+\infty,
}
which proves that $\QA{f}$ is a weak-Hardy mean.

\subsection{\label{sec:whQA} Proof of Theorem~\ref{thm:whQA}} 
At the very beginning let us recall that all conventions and results from the previous subsection remain valid.

By Mullholand's result, as $\QA{f}|_{(0,\varepsilon)}$ is a Hardy mean, we get that there exists $\alpha<1$, and $C>0$ such that 
\Eq{*}{
\QA{f}(v) \le C \cdot \P_\alpha(v)\quad\text{ for every }\quad v \in \bigcup_{n=1}^{\infty} (0,\varepsilon)^n,}

Therefore we can put $c_f(x):=C \cdot \Hc(\P_\alpha)$ for $x \le \varepsilon$. From now on we assume that $\norm{a}_\infty>\varepsilon$.

Following the idea of \cite[Proposition~3.2]{PalPas16}, we may assume that the sequence $(a_n)$ is nonincreasing i.e. $\norm{a}_\infty=a_1$. Furthermore $a_n \le \frac1n \norm{a}_1$. 
Then $n_0$ is the smallest natural number such that 
\Eq{*}{
\QA{f} (a_1,\dots,a_n) \le \varepsilon\text{ for all }n \ge n_0.
}
Indeed, as $(a_n)$ is nonincreasing, $a_1>\varepsilon$, and quasi-artihmetic mean is strict, we obtain $a_n < \QA{f}(a_1,\dots,a_n) \le \varepsilon$ for all $n \ge  n_0$.

On the other hand, by $\norm{b}_\infty=\varepsilon$, we get
\Eq{*}{
\AM{f}_{k=1}^n (b_k) \le C \cdot \PM\alpha_{k=1}^n (b_k),\quad \text{ for all } n \in \N.
}
As $\P_\alpha$ is a Hardy mean, we obtain (in the spirit of Mulholland \cite{Mul32})
\Eq{E:x2}{
\sum_{n=n_0+1}^\infty \: \AM{f}_{k=1}^n (b_k) \le
\sum_{n=1}^\infty \AM{f}_{k=1}^n (b_k) \le 
C \cdot \sum_{n=1}^\infty \PM\alpha_{k=1}^n (b_k) \le
C \cdot \Hc(\P_\alpha) \norm{b}_1. 
}

Binding this with 
\eq{E:p2}, we obtain
\Eq{E:1}{
\sum_{n=1}^\infty \: \AM{f}_{k=1}^n (a_k) \le 
n_0 \norm{a}_\infty+C \cdot \Hc(\P_\alpha) \norm{b}_1.
}
We now estimate $n_0$. Obviously, as $(a_n)$ is nonincreasing, we have $a_k \le \norm{a}_1/k$ for all $k \in \N$. 
Thus 
\Eq{*}{
\AM{f}\limits_{k=1}^n (a_k) \le 
\AM{f}\limits_{k=1}^n \Big(\min\big(\frac{\norm{a}_1}k,\norm{a}_\infty \big)\Big).
}

Now let $u(s,t)$ ($s \ge t \ge \varepsilon$) be the smallest natural number such that 
\Eq{*}{
\AM{f}\limits_{k=1}^{u(s,t)} \Big(\min\big(\tfrac sk,t \big)\Big) \le \varepsilon.
}
We have
\Eq{*}{
\AM{f}\limits_{k=1}^n (a_k) \le \AM{f}\limits_{k=1}^{n} \Big(\min\big(\tfrac {\norm{a}_1}k,\norm{a}_\infty \big)\Big) \le \varepsilon \quad \text{ for }\quad n \ge u(\norm{a}_1,\norm{a}_\infty).
}
Thus $n_0 \le u(\norm{a}_1,\norm{a}_\infty)$.

Define a weighted quasi-arithmetic mean of two variables
\Eq{*}
{
\QA{f}\big((a_1,a_2),(w_1,w_2)\big):=f^{-1} \Big( \frac{w_1f(a_1)+w_2f(a_2)}{w_1+w_2}\Big)\:.
}
Let $K \colon I \cap (\varepsilon,+\infty) \to (0,+\infty)$ be a unique function such that
\Eq{*}{
\AM{f}\big(\big(t,\tfrac\varepsilon2\big),\big(1,\,K(t)\big)\big) = \varepsilon,\quad t \in I \cap (\varepsilon,+\infty)\:.
}
Furthermore, as $\QA{f}$ is monotone we get, for $n \ge \ceil{\tfrac{2s}{\varepsilon}}$,
\Eq{*}{
\AM{f}\limits_{k=1}^n \big(\min\big(\tfrac sk,t \big)\big) 
&\le 
\QA{f}\big(
\big(t,\tfrac\varepsilon2\big), 
\big(
\ceil{\tfrac{2s}{\varepsilon}},\:n-\ceil{\tfrac{2s}{\varepsilon}}\big)\big)\\
&=\QA{f}
\big(
\big(t,\tfrac\varepsilon2\big), 
\big(
1,\:n \ceil{\tfrac{2s}\varepsilon}^{-1}-1\big)\big).
}
Now we have, for all $n \in \N$ such that $n <  u(s,t)$ and $n \ge \ceil{\tfrac{2s}\varepsilon}$,
\Eq{*}{
\varepsilon \le \AM{f}\limits_{k=1}^n \Big(\min\big(\tfrac sk,t \big)\Big) \le 
\QA{f}
\big(
\big(t,\tfrac\varepsilon2\big), 
\big(
1,\:n \ceil{\tfrac{2s}\varepsilon}^{-1}-1\big)\big)\:.
}
Thus
\Eq{*}{
n \ceil{\tfrac{2s}\varepsilon}^{-1}-1 \le K(t) \quad \text{ for all } 
n\in \N\text{ such that }n <  u(s,t)\text{ and }n \ge \ceil{\tfrac{2s}\varepsilon}\:.
}
In particular for $n:=u(s,t)-1$, we obtain
\Eq{*}{
(u(s,t)-1) \ceil{\tfrac{2s}\varepsilon}^{-1}-1\le K(t) 
\quad\text { or }\quad
u(s,t) \le \ceil{\tfrac{2s}\varepsilon}+1
,}
which implies
\Eq{*}{
u(s,t) &\le (K(t)+1) \ceil{\frac{2s}\varepsilon}+1.
}
If we divide side-by-side by $s$ and take an upper limit as $s \to \infty$ we get
\Eq{*}{
\limsup_{s \to \infty} \frac{u(s,t)}{s} \le \frac{2(K(t)+1)}{\varepsilon}.
}
Thus, as $u$ in nondecreasing with respect to both variables, there exists a function $\Phi \colon (\varepsilon,+\infty) \to \R$ such that
\Eq{*}{
u(s,t) \le \Phi(t)\cdot s \quad\text{ for all } s \ge t\ge \varepsilon.
}
In particular 
$n_0 \le u(\norm{a}_1,\norm{a}_\infty) \le \Phi(\norm{a}_\infty) \cdot \norm{a}_1.$
Combining this inequality with \eq{E:1} and \eq{E:15}, we obtain
\Eq{*}{
\sum_{n=1}^\infty \AM{f}\limits_{k=1}^n (a_k) 
&\le n_0 \norm{a}_\infty+ C\cdot\Hc(\P_\alpha) (\varepsilon n_0+\norm{a}_1)\\
&=
\big(\norm{a}_\infty + \varepsilon C\cdot\Hc(\P_\alpha) \big)\cdot n_0+C\cdot\Hc(\P_\alpha) \cdot \norm{a}_1\\
&\le
\big(\norm{a}_\infty + \varepsilon C\cdot\Hc(\P_\alpha) \big)\cdot \Phi(\norm{a}_\infty) \cdot \norm{a}_1+C\cdot\Hc(\P_\alpha) \cdot \norm{a}_1\\
&=\Big( \big(\norm{a}_\infty + \varepsilon C\cdot\Hc(\P_\alpha) \big)\cdot \Phi(\norm{a}_\infty) + C\cdot\Hc(\P_\alpha) \Big) \cdot \norm{a}_1
}
In order to conclude the proof we take 
\Eq{*}{
c_f(x):=
\begin{cases}
C\cdot\Hc(\P_\alpha) & x \le \varepsilon, \\[2mm]
 \big( x+\varepsilon C\cdot\Hc(\P_\alpha) \big) \cdot \Phi(x) +C\cdot\Hc(\P_\alpha)  & x > \varepsilon.
\end{cases}
}
\def\cprime{$'$} \def\R{\mathbb R} \def\Z{\mathbb Z} \def\Q{\mathbb Q}
  \def\C{\mathbb C}

\end{document}